\documentclass{amsart}

\usepackage{subfigure}
\usepackage{amsmath,amssymb,amsthm,graphicx,mathdots}
\usepackage{hyperref}
\usepackage{amsrefs}

\theoremstyle{definition}
\newtheorem{theorem}{Theorem}

\newtheorem{cor}[theorem]{Corollary}

\DeclareMathOperator{\si}{si}

\DeclareMathOperator{\id}{id}

\newcommand{\bY}{\begin{Young}}
\newcommand{\eY}{\end{Young}}

\title{Factorization of banded permutations}
\author{Greta Panova}
\address{Written while the author was at Harvard University, Department of Mathematics}
\thanks{Supported by a Simons Postoctoral Fellowship, University of California Los Angeles}
\date{\today}

\begin{document}

\begin{abstract}
We consider the factorization of permutations into bandwidth 1 permutations, which are products of mutually nonadjacent simple transpositions. We exhibit an upper bound on the minimal number of such factors and thus prove a conjecture of Gilbert Strang: a banded permutation of bandwidth $w$ can be represented as the product of at most $2w-1$ permutations of bandwidth 1. An analogous result holds also for infinite and cyclically banded permutations.
 \end{abstract}
\maketitle
\textbf{Journal:} \emph{Proceedings of the American Mathematical Society}.
\section{Introduction}\label{section:intro}

Computational efficiency very often requires us to represent matrices as products of certain special, easily computable, matrices using as few factors  as possible. Matrices of bounded bandwidth are often seen in practical applications. In \cite{Strang1} and \cite{Strang2} Gilbert Strang shows that when a matrix and its inverse are of bandwidth $w$, it can always be represented as a product of $O(w^2)$ such matrices of bandwidth $w=1$. In particular this bound is independent of the size of the given matrix. He also conjectures that for permutation matrices this bound is actually $2w-1$. In this paper we will prove this conjecture. This conjecture has been proven independently later also by Albert,Li,Strang and Yu in \cite{AlbertLiStrangYu} and Ezerman and Samson in \cite{SamsonEzerman}.

A \textbf{matrix of bandwidth $w$} is a matrix $A$, whose nonzero entries lie within distance $w$ from the main diagonal: $A_{i,j}=0$ whenever $\,|\,i-j\,|\,>w$. In particular, a banded permutation matrix $P$ is a $0-1$ matrix with exactly one $1$ in each row and column and such that $P_{i,j}=0$ if $\,|\,i-j\,|\,>w$. The matrix $P$ corresponds to the permutation $\pi$ defined as $\pi_i=j$ for $P_{i,j}=1$ and vice versa. So $\pi$ is of width $w$ if $\,|\,\pi_i-i\,|\, \leq w$ for every $i$. 

Our main result is the following.
\begin{theorem}\label{mainthrm}
For any permutation $\pi$, let $M=\{\pi_j-i| i<j, \pi_i>\pi_j\}$ and let $m=\#M$ be the number of different elements in the set $M$. Then there exist at most $m$ bandwidth 1 permutations $\rho^j$, i.e. $\,|\, \rho^j_i-i\,|\,\leq 1$, such that $\pi=\rho^1\rho^2\cdots$. 
\end{theorem}
If $\pi$ is a permutation of bandwidth $w$ then for the elements of $M$ we have that $\pi_j-i<\pi_i-i\leq w$ and $\pi_j - i> \pi_j - j\geq -w$, so $M\subset \{-w+1,\ldots,w-1\}$. Thus $M$ has no more than $2w-1$ elements and the conjecture follows immediately.
\begin{cor}[\textbf{Strang's conjecture}]\label{Strangsconj}
If $\pi$ is a permutation of bandwidth $w$ then there exist at most $2w-1$ bandwidth 1 permutations whose product is $\pi$. Moreover,the bound $2w-1$ is exact.
\end{cor} 

A possible extension of banded matrices, also considered by Strang in \cite{Strang2}, are infinite and cyclically banded matrices. Cyclically banded matrices are $n\times n$ matrices $A$, such that $A_{i,j} =0$ if $w<\,|\,i-j\,|\,<n-w$. Here any matrix would have width $n/2$ and so we will require that $w\leq n/2$. In Section \ref{section:cyclic} we consider the analogous question referring to their factorization into cyclically banded matrices of bandwidth 1.

We will use the notion of reduced decomposition of a permutation and its visualization called a wiring diagram. We would like to thank Alex Postnikov for suggesting their use. The proofs will rely on the construction of special wiring diagram we call a hook wiring diagram.

\section{Hook wiring diagrams}

We will consider the simple generators of the symmetric group $S_n$ as a Coxeter group. We will represent a permutation as a certain product of such simple transpositions which will be grouped into the desired bandwidth 1 factors. 

A simple transposition $s_i=(i,i+1)$ exchanges the $i$th and $(i+1)$st element. As an element of the symmetric group $S_n$, $s_i$ is equal to the permutation $1,2,\ldots, i-1, i+1, i, i+2,\ldots, n$. Note that $i$ determines the transposition, we will say that  \textbf{$i$ is the index of the transposition $s_i$}. A reduced decomposition of $\pi$ is a product $s_{i_1}s_{i_2}\cdots s_{i_l}=\pi$ of such transpositions of minimal possible length $l$; see \cite{Coxeter} for the general facts. It follows by inspection of the possible cases that every bandwidth 1 permutation is a product of mutually nonadjacent simple transpositions $s_i$, where two transpositions are adjacent if  their indices are consecutive numbers.

A \textbf{wiring diagram} (originally appearing in \cite{Goodman}) of a reduced decomposition $s_{i_1}s_{i_2}\cdots s_{i_l}=\pi$ is a planar configuration of $n$ (pseudo-)lines $L_1,\ldots,L_n$ between two columns of the numbers $1,2,\ldots,n$ with the following properties:
\begin{itemize}
\item Line $L_i$ starts at $i$ and ends at $\pi_i$.
\item No two lines intersect more than once and no three lines intersect at a point.
\end{itemize}
 Each wiring diagram depicts a reduced decomposition $s_{i_1}s_{i_2}\cdots s_{i_l}=\pi$ via the following correspondence.
Through every intersection between the lines draw a perpendicular ``dashed'' line (as in Figure \ref{example}). Counting from left to right assign to the $k$th dashed line the simple transposition $s_{i_k}=(i_k,i_{k+1})$, whose index $i_k$ is equal to 1 plus the total number of (pseudo-)lines $L_r$ that cross that dashed line above the intersection point. 

\begin{figure}[ht!]
\centering
  \includegraphics[width=3in]{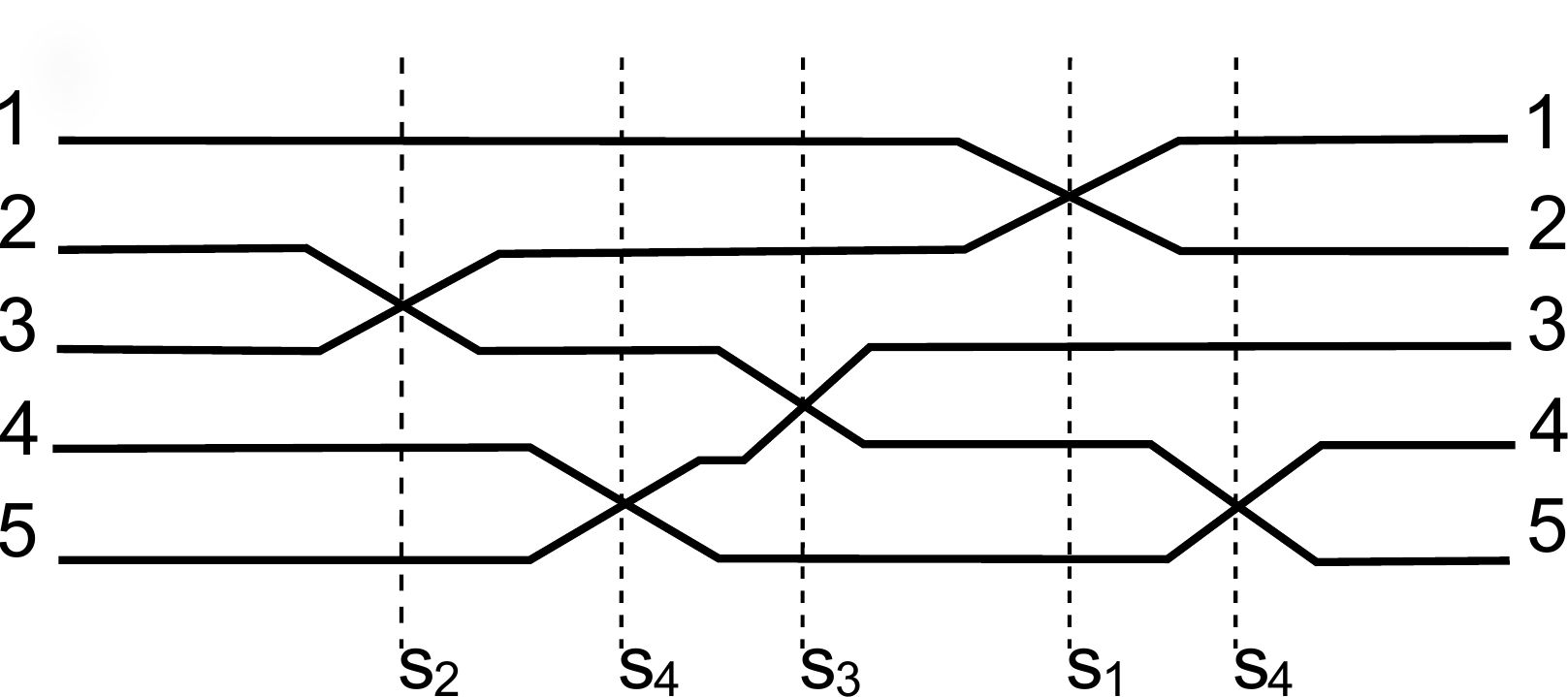}
\caption{ Wiring diagram of $\pi=s_2s_4s_3s_1s_4 = 25143$. }
\label{example}
\end{figure}

Figure \ref{example} shows the wiring diagram for $\pi = (2,3)(4,5)(3,4)(1,2)(4,5)$. Notice that line $L_i$ ``carries" the number $i$. A thin vertical slice of a wiring diagram represents an intermediate permutation with the position of $i$ being the relative vertical position of $L_i$ with respect to the other lines at this slice. Two lines crossing simply means that we exchange two adjacent entries. The number of lines vertically above that crossing plus 1 is exactly the index $i$ of the corresponding simple transposition $s_i=(i,i+1)$.

\begin{figure}[ht!]
\centering
  \includegraphics[width=2.7in]{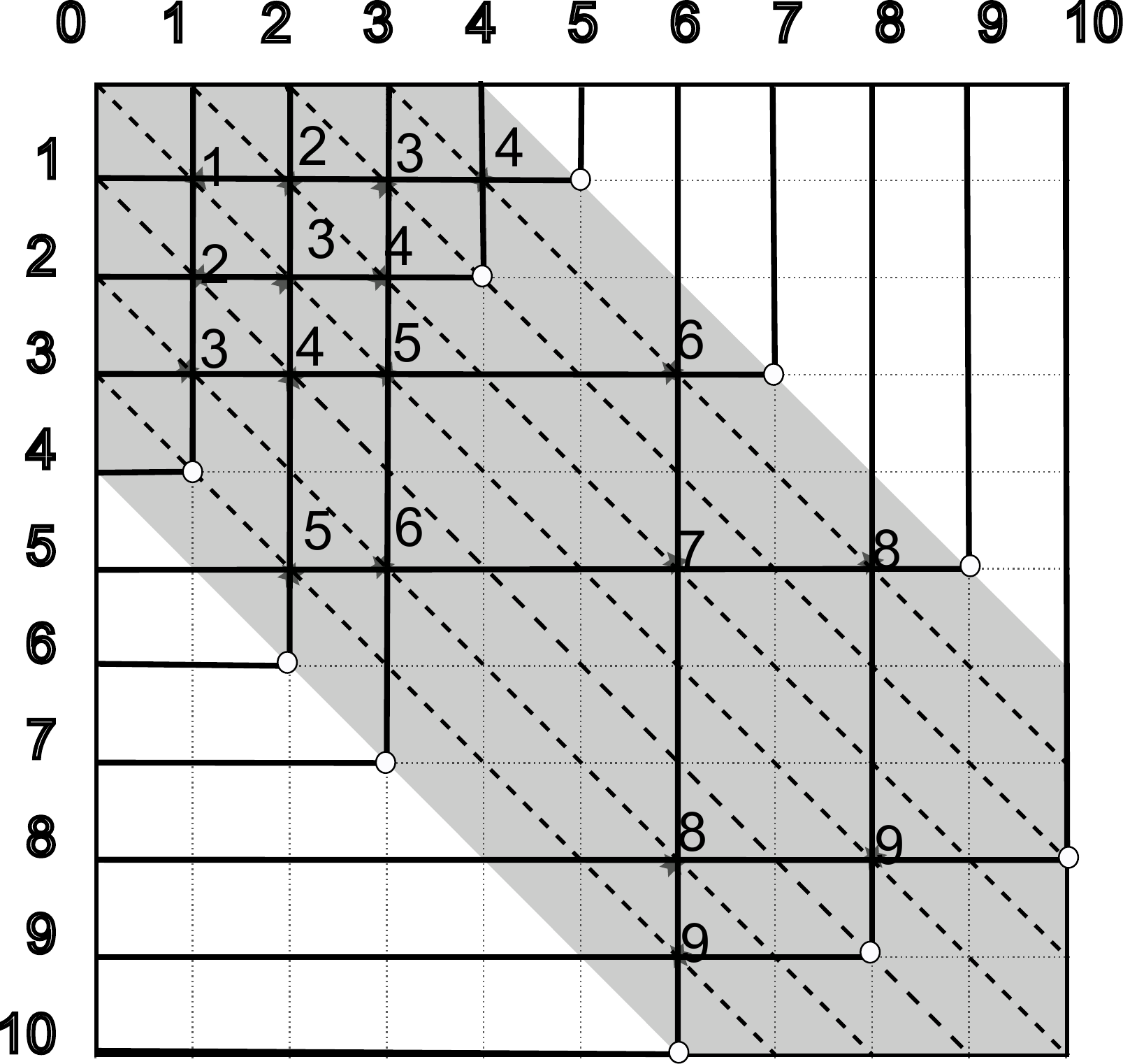} 
  \caption{Hook diagram for the permutation $\pi=5$,$4$,$7$,$1$,$9$,$2$,$3$, $10$,$8$,$6$ of bandwidth $w=4$. The dashed lines depict the diagonals $c-r=-3,-2,-1,0,1,2,3$. The numbers at the intersections are the indices of the corresponding simple transposition, e.g. $2$ corresponds to $s_2=(2,3)$.}\label{bandedexample}
\end{figure}

For any permutation $\pi$ we can also draw (see Figure \ref{bandedexample}) what we'll call a \textbf{hook diagram}. Consider a square grid bounded by $(0,0)$ in the top left corner and vertical and horizontal rays marked with $1,2,\ldots$ going down and to the right following the indexing convention for matrices; so that a point of coordinates $(r,c)$ is at the $r$-th row (counting from the top) and $c$-th column (counting left to right). Place a dot at the points $(i,\pi_i)$ on the grid and connect $(i,\pi_i)$ with $(i,0)$ and $(0,\pi_i)$ by two segments. This way the dots would be at the places of the ones in the permutation matrix of $\pi$ and each $i$ will be connected to the corresponding $\pi_i$ by a hook with corner at the dot $(i,\pi_i)$.
 
\begin{figure}[ht!]
\centering
  \includegraphics[width=1.5in]{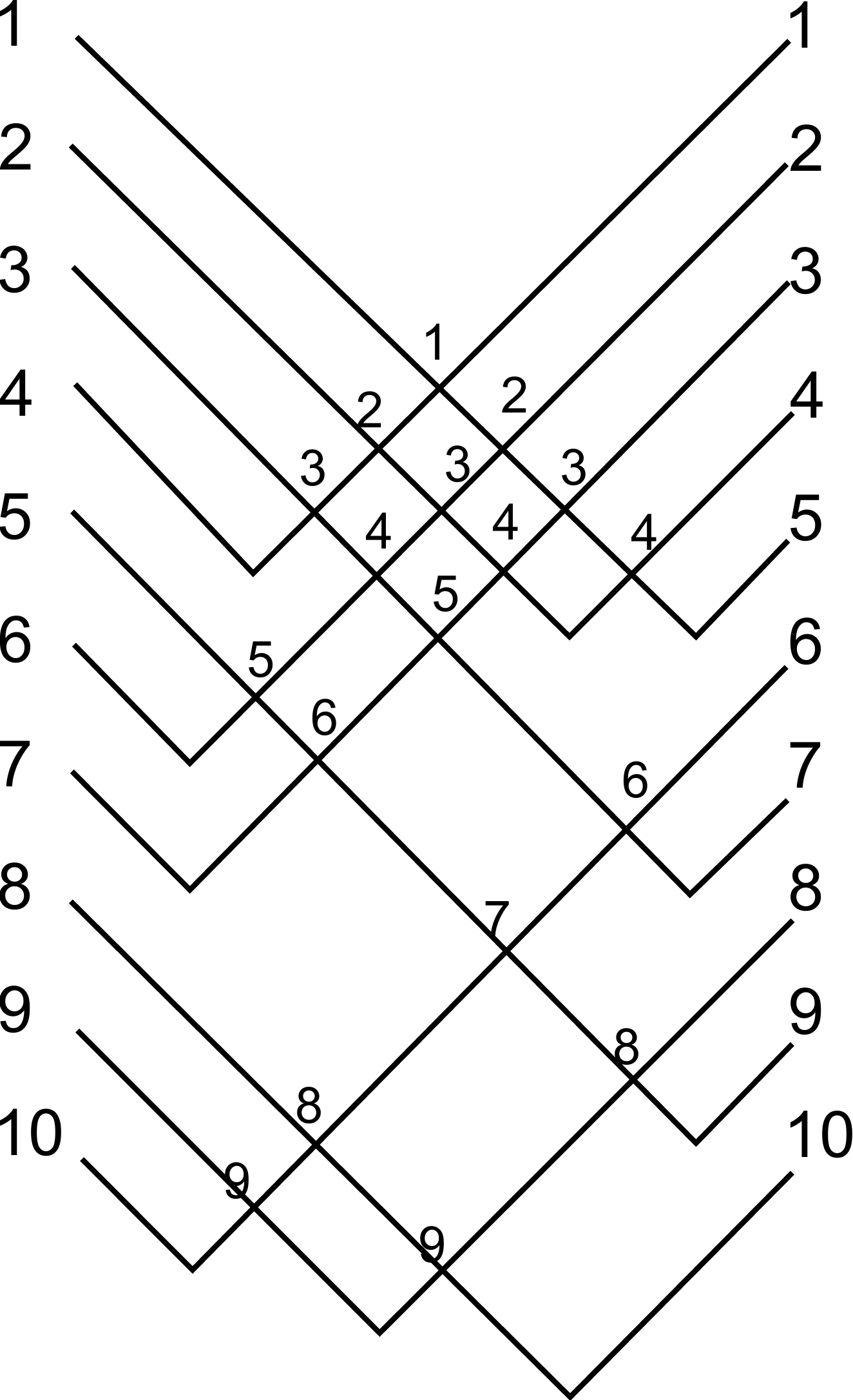}
  \caption{The wiring diagram obtained from the hook diagram for the permutation $\pi=5,4,7,1,9,2,3,10,8,6$. The numbers at the intersections indicate the indices of the corresponding simple transposition.}\label{bandedwiring}
\end{figure}

Notice that a hook diagram turns readily into a wiring diagram by extending the horizontal segments through $(0,i)$ and the vertical segments through $(j,0)$, then rotating by $-45^{\circ}$ as shown in Figure \ref{bandedwiring}. The line $L_i$ would be the rotated extended hook through the points $(0,i),(\pi_i,i),(\pi_i,0)$. To determine the index of the simple transposition corresponding to a crossing of $L_i$ and $L_j$ we need to count the number of lines in a thin strip vertically above that crossing in the rotated extended diagram. Assume $i<j$, so since $L_i$ and $L_j$ cross we must have $\pi_i>\pi_j$. The number of lines above the crossing and hence the index will be 
$r=i+\pi_j - \#\{p:p<i, \pi_p<\pi_j\}+1$, so the transposition is $s_r=(r,r+1)$. 

\begin{proof}[Proof of Theorem \ref{mainthrm}]

Draw the hook diagram of $\pi$ and interpret it as a wiring diagram as described above. 
We can now read off a reduced decomposition from the hook-wiring diagram as follows. To every intersection of two hooks assign the number $i$ of the corresponding simple transposition as explained above. 

Let $\sigma^{k}$ be the product of the transpositions on the $k$-th diagonal, i.e. $c-r=k$ (see Figure \ref{bandedexample}), where $k=-n+1,\ldots,n-1$.
We have $\sigma^{k} = s_{i_1}\cdots s_{i_l}$ where $i_1,\ldots,i_l$ are the numbers(indices) at the crossings on the $k$th diagonal. These numbers are at least 2 apart, so we have  $\sigma^{k}_i = i$ if $i\not \in \{i_1,i_1+1,\ldots,i_l,i_l+1\}$, $\sigma^{k}_{i_j}=i_j+1$ and $\sigma^{k}_{i_j+1}=i_j$ for $j=1,\ldots,l$. Then $\sigma^{k}$ is of bandwidth $1$. The diagram gives a reduced decomposition of $\pi$, where the transpositions corresponding to intersections on the same diagonal $c-r=k$ appear before the ones on the next diagonal, while the relative order of these transpositions on the same diagonal does not matter since they commute. Since their product is $\sigma^k$ we have $\pi = \sigma^{-n+1}\sigma^{-n+2}\cdots \sigma^{n-1}$. Notice that $\sigma^k$ is not trivial if and only if there is an intersection on this diagonal. There is an intersection between $L_i$ and $L_j$, $(i<j)$, if and only if $\pi_i>\pi_j$, and the intersection point is $(i,\pi_j)$. Thus diagonal $k$ has an intersection if and only if $k\in M$ and the number of nontrivial bandwidth 1 factors is equal to the cardinality of $M$, $m=\#M$. If $M=\{k_1,\ldots,k_m\}$, then the nontrivial factors are $\sigma^{k_i}$ for $i=1,\ldots,m$ and we set $\rho^i = \sigma^{k_i}$ to obtain $\pi=\rho^1\rho^2\ldots$ as desired.
\end{proof}
\begin{proof}[Proof of Corollary \ref{Strangsconj}]
As remarked earlier, if $\pi$ is banded then $M\subset\{-w+1,\ldots,w-1\}$, so $\# M\leq 2w-1$ and Strang's conjecture follows from Theorem \ref{mainthrm}.

To show that $2w-1$ is the exact bound, consider the permutation $\sigma=(w+1)(w+2)\ldots(2w)123\ldots w \ldots$ of width $w$, where the last $\ldots$ mean the identity $\sigma_i=i$ for $i>2w$. Before we show that this particular $\sigma$ cannot be factored into less than $2w-1$ permutations of bandwidth 1, we need to make a few general observations.
\begin{figure}[ht!]
\centering
 \includegraphics[height=1.5in]{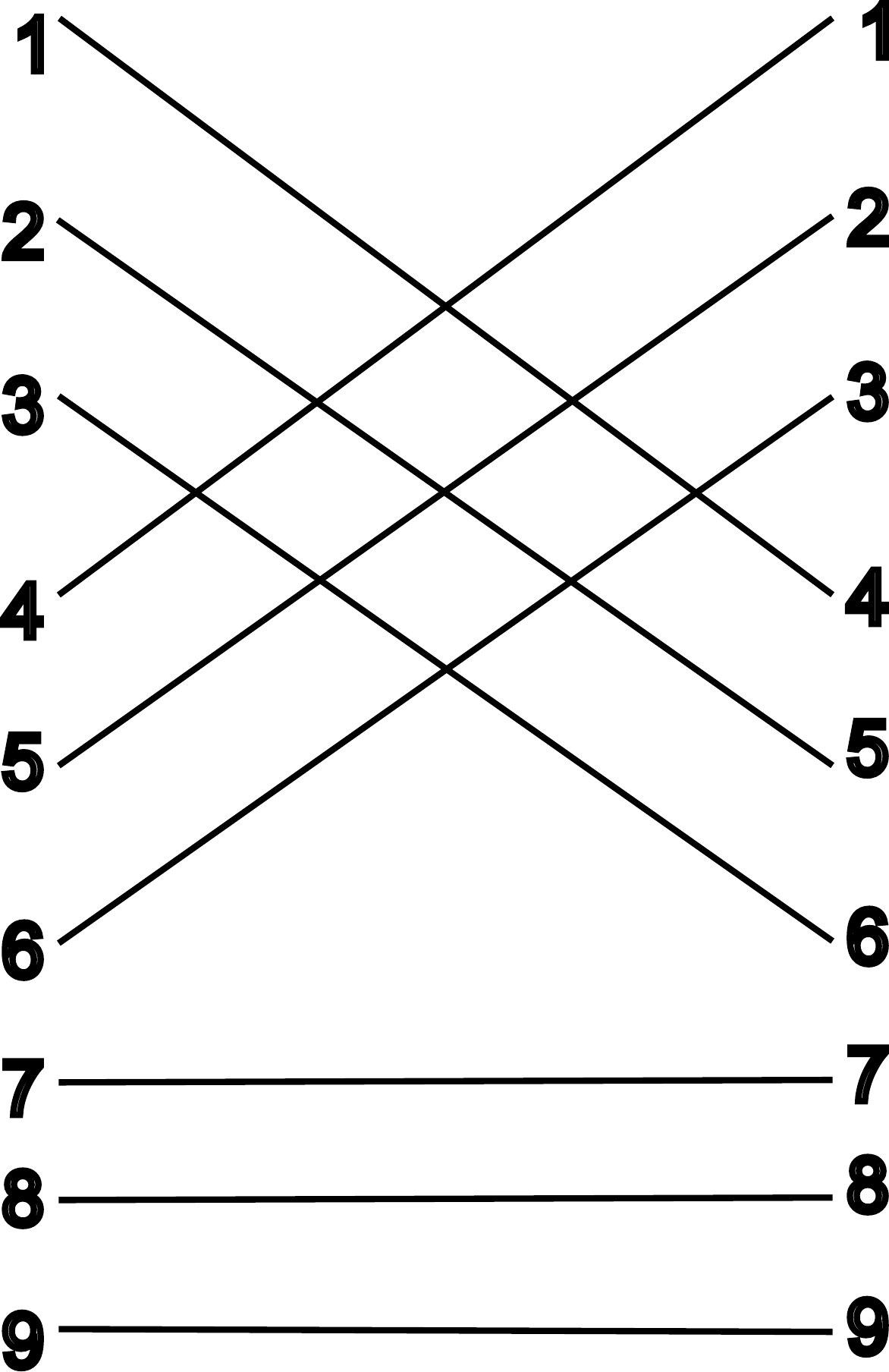}
\caption{ Any wiring diagram of $\sigma = (w+1)\ldots(2w)12\ldots w \ldots$ is homotopy equivalent to this one. Here $w=3$ and $\sigma = 456123789$.}\label{extreme}
\end{figure}

For any permutation $\pi$, let $k$ be the minimal number for which $\pi=\pi^{(1)}\cdots \pi^{(k)}$, where $\pi^{(i)}$ are permutations of bandwidth 1. Then there exists a reduced decomposition of $\pi=s^{(1)}_{i_1}s^{(1)}_{i_2}\cdots s^{(k)}_{i_l}$, such that $\pi^{(i)}$ is the product of the $i$th block of transpositions. We will show that there is always such reduced decomposition by decreasing the number of simple transpositions in it. Writing $\pi^{(i)}=s_{i_1}\cdots s_{i_m}$ as a product of transpositions we still have a decomposition of $\pi$ into simple transpositions. We can depict this decomposition graphically like a wiring diagram, without requiring that two lines intersect at most once. The assertion that the decomposition of $\pi^{(i)}=s_{i_1}\cdots s_{i_m}$ is not reduced is equivalent to two lines $L'$ and $L''$ intersecting at least twice at places $r$ and $p$ corresponding to $s_{i_r}$ and $s_{i_p}$. Let $L'=A'B'C'$ and $L''=A''B''C''$ where $A,B,C$ are the portions of the lines obtained after cutting at the two intersections. Substituting $L'$ and $L''$ with $A'B''C'$ and $A''B'C''$ respectively gives us another wiring diagram of $\pi$ for the decomposition $\pi=s_{i_1}\cdots \hat{ s_{i_r} } \cdots \hat{s_{i_p}} \cdots s_{i_m}$. Removing $s_{i_r}$ and $s_{i_p}$ from the $\pi^{(i)}$s to which they belonged gives another factorization of $\pi$ into at most $k$ permutations of width 1 with a smaller number of simple transpositions. Continuing this way we will reach the length $l$ of $\pi$ forcing the underlying decomposition into simple transpositions to be reduced. 

We can thus assume that our particular $\sigma=\sigma^{(1)}\cdots \sigma^{(k)}$ gives a reduced decomposition. Consider its wiring diagram as depicted in Figure \ref{extreme}: since the transpositions in each $\sigma^{(i)}$ are nonadjacent we can draw the corresponding intersections on the same vertical line. Thus every path from some $i$ to some $\sigma_j$ will pass through at most $k$ intersections.


Notice that any wiring diagram of $\sigma$ can be deformed (is ambiently isotopic) to a  $w\times w$ grid rotated $45^{\circ}$ like the diagram on Figure \ref{extreme}. Then every path joining $w+1$ with $\sigma_1=w+1$ has $2w-1$ intersection points and so $k\geq 2w-1$. 
\end{proof}

Since our proof is constructive, it leads  to  an algorithm for the decomposition: find the intersection points in the hook diagram and group them according to the diagonal to which they belong.

Let $I_k$ be the set of intersection points on the $k$th diagonal. Assume the inverse permutation $\pi^{-1}$ is known. Then the procedure is as follows:

\indent For $i$ from $1$ to $n$:

\indent \indent $p:=\pi_i$

\indent \indent For $j$ from $1$ to $p-1$:

\indent \indent  \indent If $\pi^{-1}_j>i$, then $I_{j-i} \leftarrow (i,j)$ 

In order to determine which transposition these intersections correspond to, notice that the number of lines $L_r$ intersecting the segment between $(i,j)$ and the origin, and thus the index of the transposition, is $i-1+j-1-\#\{t\mid t <i \,,\,\pi_t <j\}$, which we can count within this algorithm also. Let $s[i,j]=\#\{t\mid t <i \,,\,\pi_t <j\}$, then 
$$s[i,j+1]=s[i,j] + ((\pi^{-1}_j<i)),$$
where $((statement))$ denotes the logical value $0/1$ of the statement. 

\section{Infinite and cyclically banded permutations.}\label{section:cyclic}

We now consider an extension of banded matrices and their corresponding permutations $\pi=\pi_1\ldots\pi_n$, as defined in Section \ref{section:intro}. In this case bandwidth 1 encompasses more permutations and thus allowed factors.  We have a simple transposition $s_0=s_n$ exchanging $\pi_1$ and $\pi_n$, and corresponding to the bandwidth 1 cyclic matrix $A$ with $A_{1n}=A_{n1}=1$. The other additional factor is the shift $S$, acting by cyclic shift on $\pi$ as $S\pi=\pi_2\pi_3\ldots\pi_n\pi_1$, with corresponding matrix $S$ given by $S_{i,i+1}=1$ and 0 otherwise. A bandwidth 1 permutation is either a shift $S$ or $S^{-1}$ or a product of pairwise nonadjacent modulo $n$ simple transpositions $s_0,\cdots,s_{n-1}$. 

\begin{figure}[ht!]
\centering
 $A = \left[\begin{matrix} 0 &0 &0 &0 & 0& 1\\0&0&0&0&1&0\\0&1&0&0&0&0\\0&0&0&1&0&0\\0&0&1&0&0&0\\1&0&0&0&0&0\end{matrix}\right]$, 
 $B=\begin{tabular}{c@{}c@{\hspace{12pt}}cccccc@{\hspace{12pt}}c@{}c}
\dots &0 &1&0&0&0&0&0 &0     &\dots\\[10pt]

0      &1      &0&0&0&0&0&0 &0     &0\\
1      &0      &0&0&0&0&0&0 &0     &0\\
0      &0      &0&1&0&0&0&0 &0     &0\\
0      &0      &0&0&0&1&0&0 &0     &0\\
0      &0      &0&0&1&0&0&0 &0     &0\\
0      &0      &0&0&0&0&0&0 &1     &0\\[10pt]

0      &0      &0&0&0&0&0&1 &0     &0\\
\dots&0      &0&0&0&0&1&0 &0  &\dots
\end{tabular}$
\caption{ The matrix $A$ of the cyclic banded permutation $\pi = 652431$ of bandwidth $w=3$ and the corresponding infinite $B=\phi(A)$.}\label{cyclicmatrix}
\end{figure}

A cyclic banded matrix can also be interpreted as a doubly infinite periodic matrix of period $(n,n)$ in the following way.  For any $n\times n$ matrix $A$ define $\phi(A) = B$, where $B$ is a doubly infinite matrix given by 
$B_{i,j}=A_{i \pmod{n},j \pmod{n}}$ for $\,|\,i-j\,|\,\leq n/2$ and 0 otherwise, see Figure \ref{cyclicmatrix} for an example. The map $\phi$ from cyclic banded matrices of bandwidth $w$ to banded doubly infinite matrices of period $(n,n)$ is an isomorphism. We will thus consider the problem of factorization of banded doubly infinite periodic matrices and their corresponding infinite permutations.

\begin{figure}[ht!]
\centering
\includegraphics[height=2in]{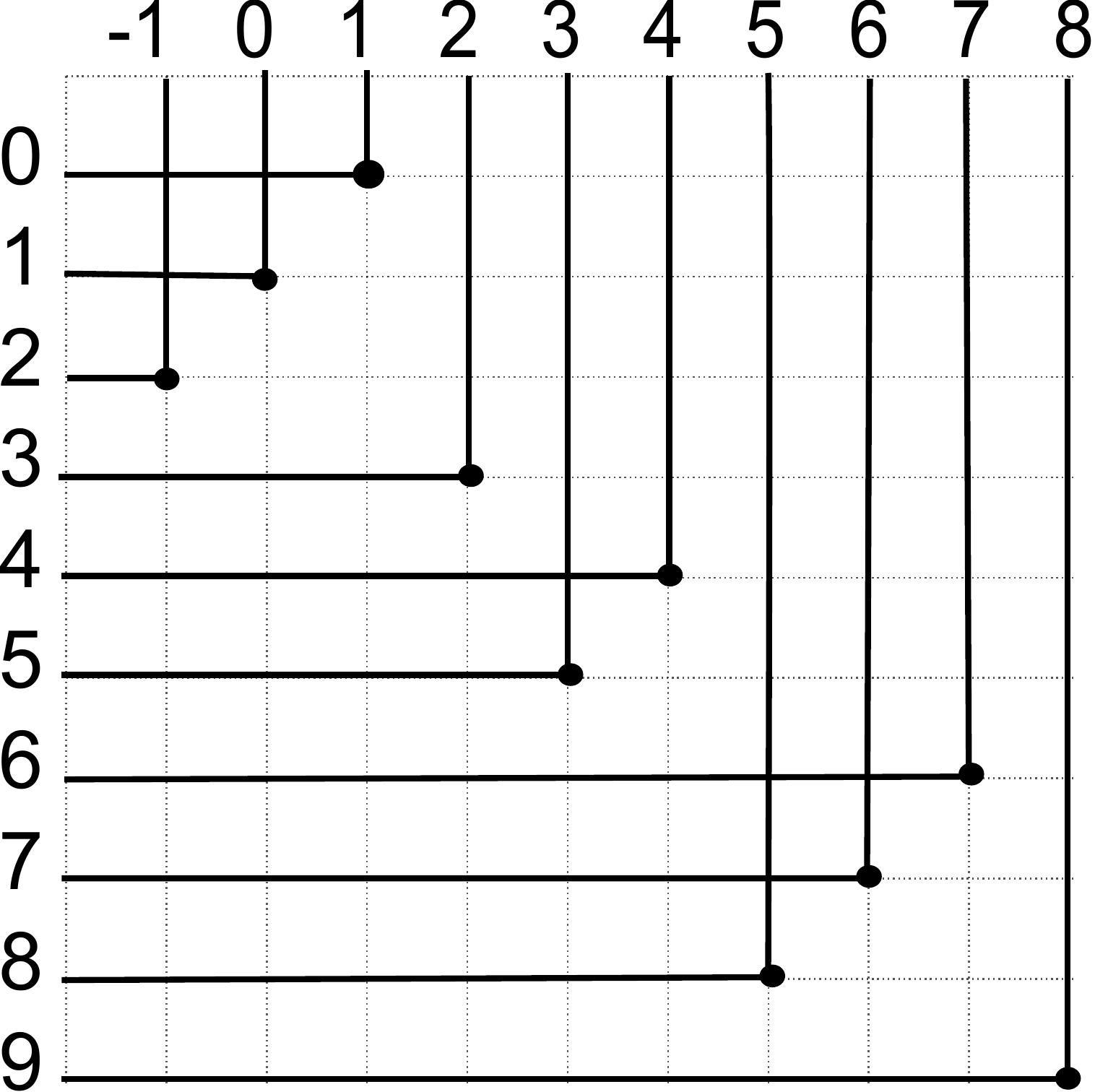} \qquad
\includegraphics[height=2in]{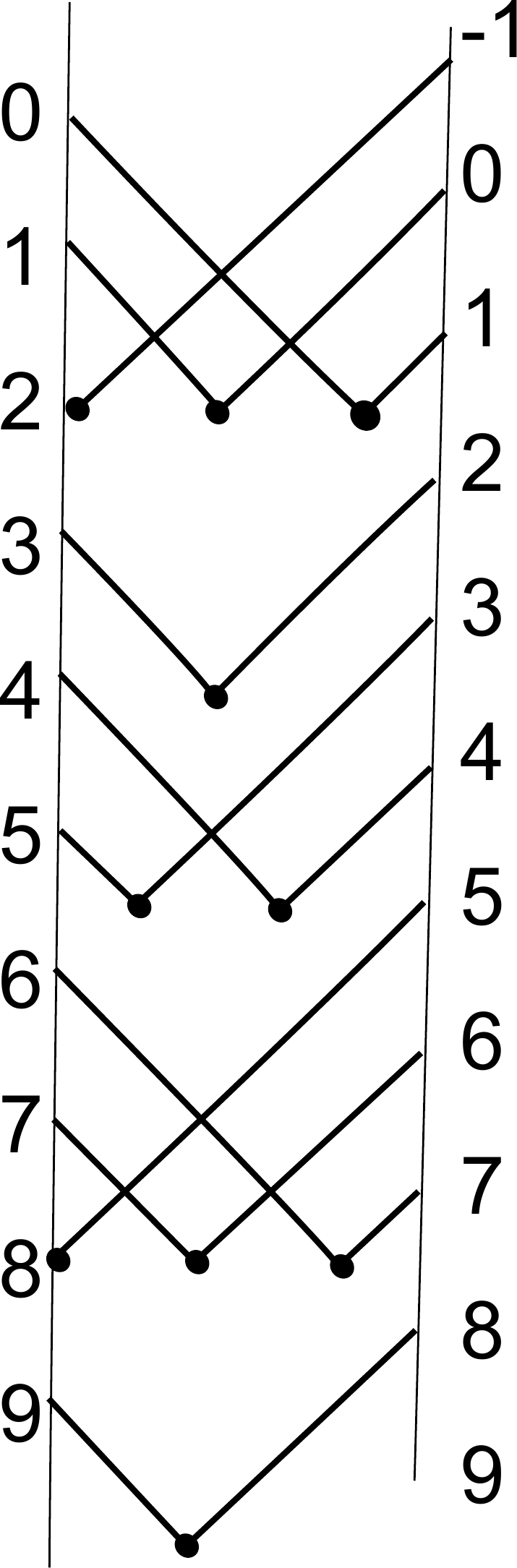}
 \qquad \includegraphics[height=2in]{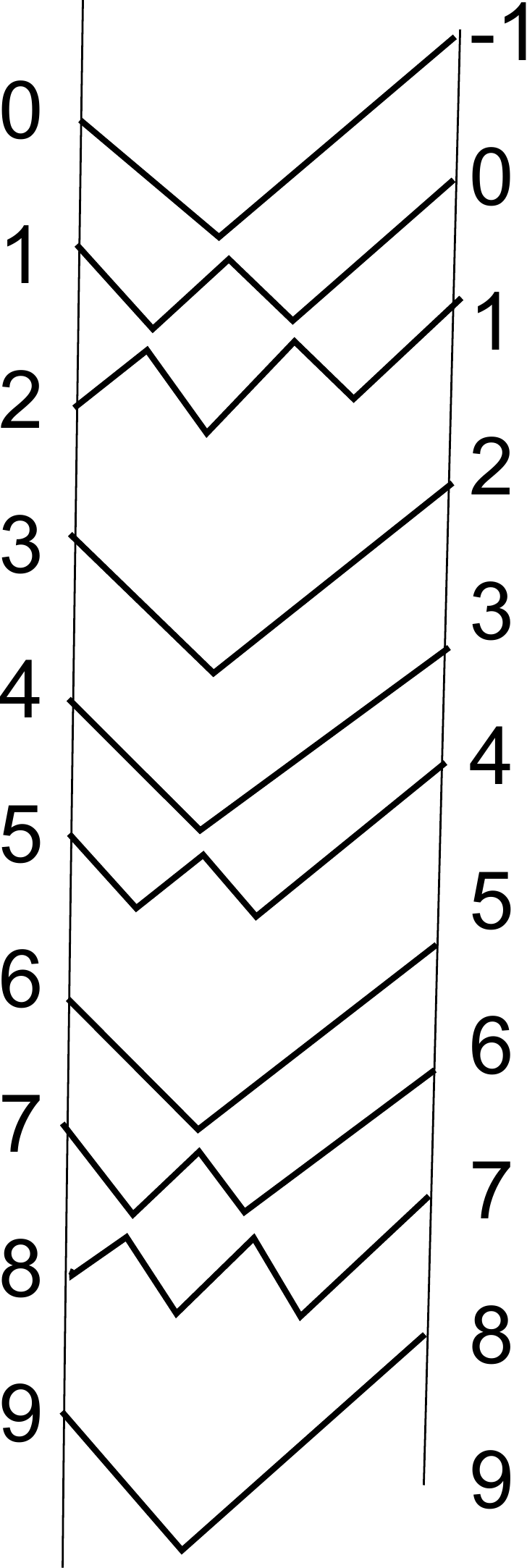}
\caption{ The hook diagram for $\pi=652431$, produced from its infinite matrix $B$, the corresponding wiring diagram which gives factors $\sigma^1=s_1, \sigma^2=s_0s_4,\sigma^3=s_1$ and the resulting trivial wiring diagram after ``untangling'' giving a shift $S^1$.}\label{cyclicwiring}
\end{figure}

We can now form the hook wiring diagram of the infinite periodic banded permutation matrix as in the previous  Section  and as it is shown on Figure 
\ref{cyclicwiring}. Again the intersections on each diagonal correspond to a bandwidth 1 factor and we can proceed to ``untangle'' them, i.e. multiply by simple transpositions. Let $\sigma^i$ be the product of the simple transpositions corresponding to intersections on the $i$th diagonal, where each transposition is represented by its infinite periodic matrix and thus the number of these transpositions will be finite. After multiplication by $\sigma^1\cdots\sigma^{2w-1}$ we will have a trivial wiring diagram where no lines intersect. Unlike in the finite case where a monotone bijective map $[1,\ldots,n]\rightarrow [1,\ldots,n]$ must be the identity, in the infinite case of $\mathbb{Z}$ these could be shifts, so the permutation/matrix corresponding to a trivial wiring diagram will be $S^k$ for some $k$. We will now show that $k\leq w$.


Define the 
\textbf{shifting index} of an infinite periodic permutation matrix $B$ as follows. Let $p:= \# \{ B_{i,j}=1\,|\,i \leq n \text{ and } n+1\leq j \}$ and $q:=\# \{ B_{i,j}=1\,|\, n+1\leq i \text{ and }  j \leq n\}$. Alternatively these are the the number of ones in the upper right $w\times w$ triangle of the original $A$ and the number of ones in the lower left such triangle.
Let the \textbf{shifting index} of $B$ be $\si(B) := p-q$. We have that $\si(\id)=0$, $\si(S)=1$, $\si(S^{-1})=-1$. We also have that $\si(SB)=1+
\si(B)$ and $\si(S^{-1}B)=-1+\si(B)$, since $S$ acts on $B$ by shifting one row upwards and the entry 1 on row $n+1$ moving to row  $n$ either decreases $q$ by 1 or increases $p$ by 1. Thus every permutation matrix factorizes uniquely as $B = S^{\si(B)}\bar{B}$, where $\si(\bar{B})=0$. Moreover, for any simple transposition $s_i$ and its corresponding infinite matrix $E_i$, we have $\si(E_i)=0$ and $\si(E_iB)=\si(B)$. This can be checked by inspection, $i=0\pmod n$ is the only nontrivial case, where it is still obvious that $p-q$ is preserved after switching row $n$ and $n+1$ in $B$. Also $SE_iS^{-1} = E_{i-1}$, where the indexing is modulo $n$ again. 

Using the usual inversion index we can show by induction on it that every such $B$ is the product of shifts and simple transpositions (note that any permutation matrix without inversions is a diagonal of 1s, i.e. $S^k$ for some $k$).
Thus we can always write $B = S^{\si(B)}E_{i_1}\cdots E_{i_l}$. This also shows that $\si(AB) =\si(A)+\si(B)$, under repeated application of $E_iS=SE_{i+1}$.   

For a matrix $B$ of bandwidth $w$ we have that $p\leq w$ and $q \leq w$, thus $-w\leq \si(B) \leq w$. Moreover, $B=\sigma^1\cdots\sigma^{2w-1} S^k$, so $\si(B) = \si(\sigma^1)+\cdots+\si(\sigma^{2w-1})+k=k$,  and thus $\,|\,k\,|\,\leq w$. We thus have the following analogue of Strang's conjecture, Corollary \ref{Strangsconj}.


\begin{theorem}\label{cyclic}
Let $\pi \in S_n$ be a cyclic banded permutation of bandwidth $w$, i.e. $\,|\,\pi_i-i\,|\,\leq w$ or $\,|\,\pi_i-i\,|\,\geq n-w$, or alternatively an infinite periodic permutation whose matrix is doubly infinite periodic with only nonzero entries in $\,|\,i-j\,|\,\leq w$. Then $\pi =\sigma^1\ldots\sigma^{2w-1}S^k$, where $\sigma^i$ is a cyclic permutation of bandwidth 1, product of nonadjacent simple transpositions, and $S$ is the cyclic shift by 1 with $\,|\,k\,|\,\leq w$.
\end{theorem}
%
%

\vfil

\end{document}